\documentclass[10pt,oneside]{article}
\Large
\usepackage{amssymb}
\usepackage{amsmath}
\usepackage{array}
\usepackage{enumerate}
\usepackage{graphicx}
\usepackage{longtable}
\usepackage{amsmath}
\usepackage{amsthm}
\usepackage{mathrsfs}
\usepackage{fancyhdr}
\usepackage{pstricks}
\usepackage{setspace}
\usepackage{graphicx}
\newtheorem{thm}{Theorem}

\begin{document}
\title{\bf \large Degree Associated Edge Reconstruction Parameters of \\Strong Double Brooms}
\author{\normalsize P. Anusha Devi and S. Monikandan \\ [-0.2cm]
\normalsize Department of Mathematics\\[-0.2cm]
\normalsize Manonmaniam Sundaranar University\\[-0.2cm]
\normalsize Tirunelveli - 627 012\\[-0.2cm]
\normalsize Tamil Nadu, INDIA\\[-0.2cm]
\normalsize {\{shaanu282,~monikandans\}@gmail.com}}
\date{}
\maketitle
\vspace{-0.4cm}

\footnote{Monikandan's research is supported by the DST-SERB, Govt. of India, Grant No. EMR/2016/000157}

\begin{abstract}
\indent\indent\indent An edge deleted unlabeled subgraph of a graph $G$ is an \emph{ecard}. A \emph{da-ecard} specifies the degree of the deleted edge along with the ecard. The \emph{degree associated edge reconstruction number} of a graph $G,~ dern(G),$ is the size of the smallest collection of da-ecards of $G$ that uniquely determines $G.$  The \emph{adversary degree associated edge reconstruction number} of a graph $G,~adern(G),$ is the minimum number $k$ such that every collection of $k$ da-ecards of $G$ uniquely determines $G.$
A \emph{strong double broom} is the graph on at least $5$ vertices obtained from a union of (at least two) internally vertex disjoint paths with same ends $u$ and $v$ by appending leaves at $u$ and $v.$ In particular,  $B(n,n,mP_{k})$ is the strong double broom with $n$ leaves at both the ends $u$ and $v$ and with $m$ internally vertex disjoint paths of order $k$ joining $u$ and $v.$ We show that $dern$ of strong double brooms is $1$ or $2.$ We also determine $adern(B(n,n,mP_{k})).$  It is 3 in most of the cases and 1 or 2 for all the remaining cases, except {$adern(B(1,1,2P_{k}))=5$} for $k\geq 4.$
\end{abstract}



\begin{section}
{Introduction}
\end{section}
		
\indent \indent All graphs considered in this paper are finite, simple and undirected.
We shall mostly follow the graph theoretic terminology of \cite{Htext}. A \emph{vertex-deleted subgraph} or \emph{card} $G-v$ of a graph (digraph) $G$ is the unlabeled graph (digraph) obtained from $G$ by deleting the vertex $v$ and all edges (arcs)  incident with $v$. The \emph{deck} of a graph (digraph) $G$ is its collection of cards. Following the formulation in [2], a graph (digraph) $G$ is \emph{reconstructible} if it can be uniquely determined from its deck. The well-known  Reconstruction Conjecture (RC) due Kelly \cite{kelly1} and Ulam \cite{ulam} asserts that every graph with at least three vertices is reconstructible. The conjecture has been proved for many
special classes, and many properties of $G$ may be deduced from its deck. Nevertheless, the full conjecture remains open. Surveys of results on the RC and related problems include \cite{bondy2,manvel}. Harary and Plantholt \cite{HPrn} defined the reconstruction number of a graph $G,$ denoted by $rn(G),$ to be the minimum number of cards which can only belong to {the deck of $G$} and not to the deck of any other graph $H,~H\ncong G,$ {these cards} thus uniquely identifying $G.$ Reconstruction numbers are known for only few classes of graphs \cite{ sur1}.\\
 \indent\indent An extension of the RC to digraphs is the \textit{Digraph Reconstruction Conjecture} (DRC), proposed by Harary \cite{Hrec}, which asserts that every digraph with at least seven vertices is reconstructible. The DRC was disproved by Stockmeyer \cite{Stoc} by exhibiting several infinite families of counter-examples and this made people doubt the RC itself. To overcome this, Ramachandran \cite{Rdi1} introduced degree associated reconstruction for digraphs and  proposed  a new conjecture in 1981. It was  proved  \cite{Rdi1}  that  the  digraphs  in  all  these  counterexamples to the DRC  obey the new conjecture, thereby protecting the RC from the threat posed by these digraph counterexamples. \\
\indent\indent The ordered triple {$(a,b,c)$ where $a,~b$ and $c$} are respectively the number of unpaired outarcs, unpaired inarcs and symmetric pair of arcs incident with $v$ in a digraph $D$ is called the \emph{degree triple of} $v.$ The \emph{degree associated card} or \emph{dacard} of a digraph (graph) is a pair $(d,C)$ consisting of a card $C$ and the degree triple (degree) $d$ of the deleted vertex. The \emph{dadeck} of a digraph is the multiset of all its dacards. 
A digraph is said to be \emph{N-reconstructible} if it can be uniquely determined from its dadeck. The \textit{new digraph reconstruction conjecture} \cite{Rdi1} (NDRC) asserts that all digraphs are N-reconstructible. 
    Ramachandran \cite{Rdrn1,Rdrn2} then studied the degree associated reconstruction number of graphs and digraphs in 2000. The \textit{degree} (\textit{degree triple}) \textit{associated reconstruction number} of a graph (digraph) $D$ is the size of the smallest collection of dacards of $D$ that uniquely determines $D.$ Articles \cite{AM2},~\cite{AM3},~\cite{AM4},~\cite{west} and \cite{m2} are recent papers on the degree associated reconstruction number.\\
 \indent\indent The edge card, edge deck, edge reconstructible graphs and edge reconstruction number are defined similarly with edge deletions instead of vertex deletions. The \textit{ edge reconstruction conjecture}, proposed by Harary \cite{Hrec}, states that all graphs with at least 4 edges are edge reconstructible.  The ordered pair $(d(e), G-e)$ is called a \emph{degree associated edge card} or \emph{da-ecard} of the graph $G,$ where $d(e)$ (called the \emph{degree} of $e$) is the number of edges adjacent to $e$ in $G.$ The \emph{edeck} (\emph{da-edeck}) of a graph $G$ is its collection of ecards (da-ecards). For an edge reconstructible graph $G,$ Molina studied \cite{Molerndis} the \emph{edge reconstruction number} of $G,$ which is defined to be the size of the smallest subcollection of the edeck of $G$ which is not contained in the edeck of any other graph $H,$ $H \not\cong G$. For an edge reconstructible graph $G$ from its da-edeck, the \emph{degree associated edge reconstruction number} of a graph $G,$ denoted by \emph{dern}($G$), is the size of the smallest subcollection of the da-edeck of $G$ which is not contained in the da-edeck of any other graph $H,$ $H \not\cong G$. The \emph{adversary degree associated edge reconstruction number} of a graph $G,$ \emph{adern}($G$), is the minimum number $k$ such that every collection of $k$ da-ecards of $G$ is not contained in the da-edeck of any other graph $H,~H \ncong G$. Degree associated edge reconstruction parameters might be a strong tool for providing evidence to support or reject the Edge Reconstruction Conjecture that remains open. For very few classes of graphs, these edge reconstruction parameters have been determined \cite{AM1, m1, m2, MSr, Iwoca}.\\
\indent A vertex of degree $m$ is called an \emph{$m$-vertex,} and a 1-vertex is called an \emph{end vertex}.
The neighbour of a 1-vertex is called a \emph{base,} and a base of degree $m$ is called an \emph{$m$-base}. 
A neighbour of $v$ with degree $k$ is called a $k$-neighbour of $v.$
A \emph{double broom} is a tree obtained from a path by appending leaves at both ends of the path.
A \emph{strong double broom}, denoted by $B,$ is the graph on at least $5$ vertices obtained from a union of (at least two) internally vertex disjoint $(u,v)$-paths by appending leaves at $u$ and $v.$
More precisely, $B(n_1,n_2,m_1P_{k_1},m_2P_{k_2},\ldots,m_tP_{k_t})$ denotes the strong double broom with $n_1$ leaves at one end $u$ and $n_2$ leaves at the other end $v$ and there are $m_i$ internally vertex disjoint $(u,v)$-paths on $k_i$ vertices for $1\leq i \leq t,~m_i\geq0$ and $k_1<k_2<\ldots <k_t.$
The vertices $u$ and $v$ are called the \emph{hub vertices} and the $2$-vertices are called the \emph{middle vertices}.
It is clear that $m_1=1$ when $k_1=2.$\\
\indent   Recently Ma et al. \cite{m2} have determined $adrn$ of double brooms. In this paper, we determine $dern$ and $adern$ of strong double brooms. We show that $dern(B(n_1,n_2,m_1P_{k_1},m_2P_{k_2},\ldots,m_tP_{k_t}))$ is $1$ or $2$  and that $adern(B(n,n,mP_{k}))$ is 3 in most of the cases. For all the exceptional cases, usually $adern(B(n,n,mP_{k}))$ is 1 or 2, except {$adern(B(1,1,2P_{k}))=5$} for $k\geq 4.$
\section{Dern of Strong Double Brooms}
\indent\indent  The da-ecards of $B$ are classified into three types: a leaf da-ecard $L,$ a middle da-ecard $M$ and a hub da-ecard $K$ are obtained, respectively, by deleting an edge incident to a leaf vertex and a hub vertex, an edge of degree sum $2$ and an edge incident to a hub vertex and a $2$-vertex (Figure~1).
\begin{center}
\scalebox{1} 
{
\begin{pspicture}(0,-2.5789065)(7.7433677,2.5789065)
\psbezier[linewidth=0.04](2.169058,-0.0510935)(2.169058,-0.9272839)(5.6490564,-0.9710935)(5.6490564,-0.0949029)(5.6490564,0.7812877)(2.169058,0.8250969)(2.169058,-0.0510935)
\psbezier[linewidth=0.04](2.2490578,-0.079874)(2.2490578,-0.4310935)(5.569058,-0.4223129)(5.569058,-0.0710935)(5.569058,0.2801259)(2.2490578,0.2713455)(2.2490578,-0.079874)
\psline[linewidth=0.04cm,linestyle=dotted,dotsep=0.16cm](3.7890575,0.14890654)(3.8090577,-0.3110935)
\psline[linewidth=0.04cm](1.4290578,-0.8110935)(2.169058,-0.1110934)
\psline[linewidth=0.04cm](1.4490578,-0.3310935)(2.169058,-0.0910935)
\psline[linewidth=0.04cm](1.4290578,0.1089066)(2.149057,-0.0510935)
\psline[linewidth=0.04cm](5.6290574,-0.0710935)(6.3290577,0.1089066)
\psline[linewidth=0.04cm](5.6490564,-0.0710935)(6.3690577,-0.3310935)
\psline[linewidth=0.04cm](5.6690574,-0.1110934)(6.3490577,-0.8110935)
\psline[linewidth=0.04cm,linestyle=dotted,dotsep=0.16cm](6.3890576,0.7689066)(6.3890576,0.2689066)
\psline[linewidth=0.04cm,linestyle=dotted,dotsep=0.16cm](1.4090577,0.74890655)(1.4090577,0.24890655)
\psline[linewidth=0.04cm](1.4490578,0.8689066)(2.149057,-0.0510935)
\psline[linewidth=0.04cm](5.6690574,-0.03109347)(6.3890576,0.9289065)
\pscustom[linewidth=0.04]
{
\newpath
\moveto(6.7290573,1.0489063)
\lineto(6.8090587,1.0589067)
\curveto(6.849059,1.0639068)(6.8990593,1.0489068)(6.909058,1.0289065)
\curveto(6.919057,1.0089062)(6.929057,0.9439062)(6.9290586,0.8989065)
\curveto(6.9290605,0.8539068)(6.9290605,0.7839068)(6.9290586,0.7589065)
\curveto(6.929057,0.7339062)(6.929057,0.6639062)(6.9290586,0.6189065)
\curveto(6.9290605,0.5739068)(6.9290605,0.4989068)(6.9290586,0.4689065)
\curveto(6.929057,0.43890634)(6.929057,0.3639065)(6.9290586,0.3189065)
\curveto(6.9290605,0.2739065)(6.9440603,0.2089065)(6.959059,0.1889065)
\curveto(6.9740567,0.1689065)(7.004057,0.1289065)(7.0190587,0.1089065)
\curveto(7.0340605,0.0889065)(7.06906,0.0639065)(7.089058,0.0589065)
\curveto(7.1090555,0.0539065)(7.1090555,0.0289065)(7.089058,0.0089065)
\curveto(7.06906,-0.0110935)(7.03906,-0.0610935)(7.029058,-0.0910935)
\curveto(7.019056,-0.1210935)(7.0090556,-0.1860935)(7.009058,-0.2210935)
\curveto(7.00906,-0.2560935)(7.00906,-0.3210935)(7.009058,-0.3510935)
\curveto(7.0090556,-0.3810935)(7.0090556,-0.4360935)(7.009058,-0.4610935)
\curveto(7.00906,-0.4860935)(7.00906,-0.5360935)(7.009058,-0.5610935)
\curveto(7.0090556,-0.5860935)(7.0090556,-0.6360935)(7.009058,-0.6610935)
\curveto(7.00906,-0.6860935)(7.00906,-0.7360935)(7.009058,-0.7610935)
\curveto(7.0090556,-0.7860935)(7.0090556,-0.8360935)(7.009058,-0.8610935)
\curveto(7.00906,-0.8860935)(6.99406,-0.9260935)(6.9790583,-0.9410935)
\curveto(6.9640555,-0.9560935)(6.9240556,-0.9760935)(6.8990574,-0.9810935)
\curveto(6.874059,-0.9860935)(6.81906,-0.9710935)(6.7890587,-0.9510935)
}
\pscustom[linewidth=0.04]
{
\newpath
\moveto(1.008686,-1.0578324)
\lineto(0.9284168,-1.0653757)
\curveto(0.8882822,-1.0691473)(0.8387656,-1.0526221)(0.8293836,-1.0323251)
\curveto(0.82000154,-1.0120282)(0.81199867,-0.9467521)(0.81337774,-0.9017728)
\curveto(0.8147568,-0.8567936)(0.8169022,-0.78682625)(0.8176685,-0.76183814)
\curveto(0.8184348,-0.73685)(0.8205799,-0.666883)(0.82195866,-0.6219041)
\curveto(0.82333744,-0.5769254)(0.8256363,-0.5019608)(0.82655644,-0.47197485)
\curveto(0.82747656,-0.4419889)(0.82977575,-0.36702427)(0.8311548,-0.32204565)
\curveto(0.8325339,-0.277067)(0.8195328,-0.21163732)(0.80515265,-0.19118689)
\curveto(0.79077244,-0.17073645)(0.76201326,-0.12983587)(0.7476343,-0.10938574)
\curveto(0.7332553,-0.0889356)(0.6990372,-0.0628745)(0.679198,-0.057263546)
\curveto(0.65935886,-0.05165258)(0.6601255,-0.026664484)(0.68073124,-0.0072873477)
\curveto(0.7013369,0.012089789)(0.7328555,0.061146796)(0.74376833,0.09082667)
\curveto(0.7546813,0.120506845)(0.7666686,0.18516992)(0.76774293,0.22015284)
\curveto(0.7688174,0.25513604)(0.7708096,0.32010552)(0.7717273,0.35009181)
\curveto(0.7726452,0.38007838)(0.7743313,0.43505257)(0.7750995,0.46004054)
\curveto(0.7758679,0.48502862)(0.7774005,0.53500485)(0.7781647,0.5599929)
\curveto(0.77892894,0.584981)(0.78046155,0.63495785)(0.7812299,0.65994656)
\curveto(0.78199816,0.6849353)(0.7835305,0.73491144)(0.7842945,0.75989896)
\curveto(0.7850585,0.7848864)(0.78659046,0.83486325)(0.78735846,0.85985255)
\curveto(0.78812647,0.88484186)(0.804346,0.9243634)(0.81979746,0.9388955)
\curveto(0.8352489,0.9534277)(0.8758431,0.97219205)(0.9009859,0.9764242)
\curveto(0.92612857,0.9806564)(0.98064244,0.96397763)(1.0100136,0.94306666)
}
\usefont{T1}{ptm}{m}{n}
\rput{90.74481}(-0.09186524,-0.46092767){\rput(0.16249213,-0.2561157){$n_1$ vertices}}
\usefont{T1}{ptm}{m}{n}
\rput{90.74481}(7.5035334,-7.6382284){\rput(7.502492,-0.09611648){$n_2$ vertices}}
\pscustom[linewidth=0.02]
{
\newpath
\moveto(1.3890587,1.1889067)
\lineto(1.4090575,1.2289065)
\curveto(1.4190569,1.2489065)(1.4740577,1.2989064)(1.5190591,1.3289065)
\curveto(1.5640602,1.3589065)(1.6540596,1.4139067)(1.6990577,1.4389067)
\curveto(1.7440556,1.4639066)(1.8290557,1.4939065)(1.8690577,1.4989065)
\curveto(1.9090596,1.5039065)(1.9790597,1.5089065)(2.0090575,1.5089065)
}
\pscustom[linewidth=0.02]
{
\newpath
\moveto(1.9290589,1.4289063)
\lineto(1.9490577,1.4589065)
\curveto(1.9590571,1.4739065)(1.9990569,1.4989065)(2.0290575,1.5089065)
}
\pscustom[linewidth=0.02]
{
\newpath
\moveto(2.0290577,1.5089065)
\lineto(1.9890578,1.5289065)
\curveto(1.9690578,1.5389065)(1.9390584,1.5589066)(1.9290589,1.5689065)
}
\pscustom[linewidth=0.02]
{
\newpath
\moveto(5.6890574,-0.3910935)
\lineto(5.6890574,-0.4710935)
\curveto(5.6890574,-0.5110935)(5.6790576,-0.6060935)(5.6690574,-0.6610935)
\curveto(5.6590576,-0.7160935)(5.634058,-0.8110935)(5.6190586,-0.8510935)
\curveto(5.604059,-0.8910935)(5.5540595,-0.9910935)(5.5190587,-1.0510935)
\curveto(5.484058,-1.1110935)(5.449058,-1.1810935)(5.449058,-1.1910934)
}
\pscustom[linewidth=0.02]
{
\newpath
\moveto(5.4490576,-1.1710935)
\lineto(5.4990587,-1.1510935)
\curveto(5.5240593,-1.1410935)(5.579059,-1.1110935)(5.609058,-1.0910935)
}
\pscustom[linewidth=0.03]
{
\newpath
\moveto(5.4490576,-1.1510935)
\lineto(5.439057,-1.1110935)
\curveto(5.4340568,-1.0910935)(5.419057,-1.0510935)(5.4090576,-1.0310935)
}
\usefont{T1}{ptm}{m}{n}
\rput(2.1740577,1.9389066){leaf vertex}
\usefont{T1}{ptm}{m}{n}
\rput(5.2893705,-1.3410933){hub vertex}
\psbezier[linewidth=0.02](2.8243685,0.7389065)(2.6382456,0.66020477)(3.0280068,1.8073486)(4.0022235,2.0561028)
\psbezier[linewidth=0.02](3.326533,0.821795)(3.211614,0.8189065)(3.3643687,1.5835205)(4.0643687,2.1189065)
\psbezier[linewidth=0.02](4.324369,0.7789065)(4.30437,0.7589065)(3.7643688,1.10962)(4.0443697,2.1189065)
\psbezier[linewidth=0.02](5.0197763,0.7338471)(5.0861053,0.6358557)(4.0617986,1.1840669)(4.027445,2.0944884)
\usefont{T1}{ptm}{m}{n}
\rput(4.3762455,2.3989065){middle vertices}
\usefont{T1}{ptm}{m}{n}
\rput(4.0368686,-2.351094){Figure~1.~Strong Double Broom}
\pscircle[linewidth=0.04,dimen=outer,fillstyle=solid](1.4114109,0.8942325){0.08}
\pscircle[linewidth=0.04,dimen=outer,fillstyle=solid](1.3714108,0.13423252){0.08}
\pscircle[linewidth=0.04,dimen=outer,fillstyle=solid](1.391411,-0.34576747){0.08}
\pscircle[linewidth=0.04,dimen=outer,fillstyle=solid](1.4314109,-0.8057675){0.08}
\pscircle[linewidth=0.04,dimen=outer,fillstyle=solid](6.411411,0.9142325){0.08}
\pscircle[linewidth=0.04,dimen=outer,fillstyle=solid](6.391411,0.13423252){0.08}
\pscircle[linewidth=0.04,dimen=outer,fillstyle=solid](6.411411,-0.3257675){0.08}
\pscircle[linewidth=0.04,dimen=outer,fillstyle=solid](6.411411,-0.82576746){0.08}
\pscircle[linewidth=0.04,dimen=outer,fillstyle=solid](5.6114106,-0.06576748){0.08}
\pscircle[linewidth=0.04,dimen=outer,fillstyle=solid](5.231411,-0.5257675){0.08}
\pscircle[linewidth=0.04,dimen=outer,fillstyle=solid](4.411411,-0.70576745){0.08}
\pscircle[linewidth=0.04,dimen=outer,fillstyle=solid](3.5914109,-0.7257675){0.08}
\pscircle[linewidth=0.04,dimen=outer,fillstyle=solid](2.7914107,-0.56576747){0.08}
\pscircle[linewidth=0.04,dimen=outer,fillstyle=solid](2.5914109,-0.24576749){0.08}
\pscircle[linewidth=0.04,dimen=outer,fillstyle=solid](3.0514107,-0.2857675){0.08}
\pscircle[linewidth=0.04,dimen=outer,fillstyle=solid](4.0314107,-0.34576747){0.08}
\pscircle[linewidth=0.04,dimen=outer,fillstyle=solid](3.5514107,-0.3257675){0.08}
\pscircle[linewidth=0.04,dimen=outer,fillstyle=solid](4.4914107,-0.3257675){0.08}
\pscircle[linewidth=0.04,dimen=outer,fillstyle=solid](5.0314107,-0.26576748){0.08}
\pscircle[linewidth=0.04,dimen=outer,fillstyle=solid](4.731411,0.15423252){0.08}
\pscircle[linewidth=0.04,dimen=outer,fillstyle=solid](3.9114108,0.19423252){0.08}
\pscircle[linewidth=0.04,dimen=outer,fillstyle=solid](3.071411,0.13423252){0.08}
\pscircle[linewidth=0.04,dimen=outer,fillstyle=solid](2.5514107,0.37423253){0.08}
\pscircle[linewidth=0.04,dimen=outer,fillstyle=solid](3.111411,0.5342325){0.08}
\pscircle[linewidth=0.04,dimen=outer,fillstyle=solid](3.6914108,0.5942325){0.08}
\pscircle[linewidth=0.04,dimen=outer,fillstyle=solid](4.251411,0.5742325){0.08}
\pscircle[linewidth=0.04,dimen=outer,fillstyle=solid](4.811411,0.4942325){0.08}
\pscircle[linewidth=0.04,dimen=outer,fillstyle=solid](5.291411,0.35423252){0.08}
\pscircle[linewidth=0.04,dimen=outer,fillstyle=solid](2.2114108,-0.0657675){0.08}
\end{pspicture} 
}
\end{center}

\indent \indent An \emph{extension} of a da-ecard $(d(e),G-e)$ of $G$ is a graph obtained from the da-ecard by adding a new edge joining two non adjacent vertices whose degree sum is $d(e)$ and it is denoted by $H(d(e),G-e)$ (or simply $H$). Throughout this paper, $H$ and $e$ are used in the sense of this definition. In the proof of every theorem, $G$ denotes the strong double broom considered in that theorem.\\
\begin{thm}\label{1}
$dern(B(n,n,mP_k))=
									\begin{cases}
									1& \textit{if $n+m\geq6$ or `~$n+m<6,(n,m)\neq(1,2),(1,3)$ and $k=3$'}\\
									2& \textit{otherwise}\\
									\end{cases}$
									\end{thm}
\begin{proof}
For $k=2,~m$ would be 1 (as $G$ is simple) which is excluded in the definition itself.
Therefore $k>2.$\\
\textit{Case~$1.$}~$n+m\geq6.$\\
\indent In any leaf da-ecard $(n+m-1,L),$ exactly two vertices have degree sum $n+m-1$ and hence $G$ can be obtained uniquely from the da-ecard $(n+m-1,L)$ by adding an edge joining the unique isolated vertex and the unique $(n+m-1)$-vertex and thus $dern(G)=1$ in this case.\\ 
\textit{Case~$2.$}~$n+m<6,~(n,m)\neq(1,2),(1,3)$ and $k=3.$\\
\indent In any hub da-ecard $(n+m,K),$ the degrees of the vertices are $1,~2,~n+m-1$ and $n+m;$ in fact only one vertex, say $v$ has degree $n+m-1.$ Hence all simple graphs obtained by adding an edge joining $v$ and any end vertex are isomorphic and they are $G.$ Therefore $dern(G)=1.$\\
\textit{Case~$3.$}~`$n+m<6$ and $k\geq4$' or `$(n,m)=(1,2),(1,3)$ and $k=3$'.\\
\indent Now consider a leaf da-ecard $(n+m-1,L).$ It clearly forces $G$ to have $m\choose2$ cycles of length $2k-2.$
For $k\geq4,$ consider the middle da-ecard $(2,M)$ in addition with $L.$
Then all extensions obtained from $M$ by adding a new edge joining two 1-vertices at distance $2k-3$ are isomorphic and they are $G.$
For $k=3,$ consider the hub da-ecard $(n+m,K)$ in addition with $L.$
Now all extensions obtained from $K$ by adding a new edge joining the unique  $(n+m-1)$-base and any one of the two 1-vertices  at distance $3$ are isomorphic and they are $G.$
Therefore $dern(G)=2.$
\end{proof}
\begin{thm}\label{3}
For $n_1<n_2,$
\begin{equation*}dern(B(n_1,n_2,mP_k))=
									\begin{cases}
									2& \textit{if `$n_1+m\leq5,~n_2-n_1=2$ or $3$ and $k>3$'}\\
									~& \textit{or `$n_1+m\leq5,~n_2-n_1=1,~n_2+m\neq6$ and $k>3$'}\\
									1& \textit{otherwise}\\
									\end{cases}
									\end{equation*}
\end{thm}
\begin{proof}
For $k=3,$ any hub da-ecard $(n_2+m,K)$ uniquely determines $G.$
So consider $k>3.$\\
\textit{Case~$1.$}~$n_1+m\leq5.$\\
\textit{Case~$1.1.$}~`$n_2-n_1=2$ or $3$' or `$n_2-n_1=1$ and $n_2+m\neq6$'.\\
\indent Any leaf da-ecard forces $G$ to have ${m\choose2}$ cycles of length $2k-2.$
Hence in any middle da-ecard, the newly added edge must be incident to two 1-vertices at distance $2k-3$ and the resulting graph thus obtained is isomorphic to $G$ and hence $dern(G)=2.$\\
\textit{Case~$1.2.$}~`$n_2-n_1\geq4$' or `$n_2-n_1=1$ and $n_2+m=6$'.\\
\indent Any leaf da-ecard $(n_2+m-1,L)$ uniquely determines $G.$\\
\textit{Case~$2.$}~$n_1+m\geq6.$\\
\indent Here any leaf da-ecard $(n_1+m-1,L)$ will determine $G$ uniquely.
\end{proof}
\begin{thm}\label{4'}
$dern(B(n,n,m_1P_{k_1},m_2P_{k_2},\dots,m_tP_{k_t}))=	
						\begin{cases}
						1& \textit{if $n+m\geq6$ or $n+m=5,~k_1=3$}\\
						~& \textit{$n+m=4,~k_1=k_2-1=3$ and $n=2$}\\
						2& \textit{otherwise}\\
						\end{cases}$
\end{thm}
\begin{proof}
In view of Theorem \ref{1}, we can assume that at least two $m_i$'s are nonzero.
For $n+m\geq6,$ any leaf da-ecard uniquely determines $G$ and hence $dern(G)=1.$
If either `$n+m=5$ and $k_1=3$' or `$n+m=4,~k_1=k_2-1=3$ and $n=2',$ then any hub da-ecard, obtained by removing a hub edge lying on a path of length 3, uniquely determines $G$ and so $dern(G)=1.$\\
\indent For all the remaining cases, we use a leaf da-ecard and a middle da-ecard (obtained by removing an edge lying on a path of length $k_t$) to determine $G.$ 
The leaf da-ecard forces $G$ to have ${m_1\choose2}$ cycles of length $2(k_1-1)$ (if $k_1>2$), ${m_2\choose2}$ cycles of length $2(k_2-1)$ (if $k_2>2$),~\ldots,${m_t\choose2}$ cycles of length $2(k_t-1)$ (if $k_t>2$) and to have $m_1m_2$ cycles of length $k_1+k_2-2,~m_1m_3$ cycles of length $k_1+k_3-2,\ldots,m_1m_t$ cycles of length $k_1+ k_t-2,~m_2m_3$ cycles of length $k_2+k_3-2,\ldots,m_2m_t$ cycles of length $k_2+k_t-2,\ldots,m_{t-1}m_t$ cycles of length $k_{t-1}+k_t-2.$
Hence in the middle da-ecard, the newly added edge must be incident to two 1-vertices at distance $k_1+k_t-3$ and hence $dern(G)=2.$
\end{proof}
\begin{thm}\label{4}
\begin{equation*}dern(B(n_1,n_2,P_2,m_2P_{k_2},\ldots,m_tP_{k_t}))=
									\begin{cases}
									2& \textit{if $\sum\limits_{i=2}^t m_i=1,~k_j>3~(2\leq j\leq t),~n_1=1,~n_2\leq4$}\\
								1& \textit{otherwise}\\
									\end{cases}
									\end{equation*}
\end{thm}
\begin{proof}
\indent Assume $\sum\limits_{i=2}^t m_i$ is 1 as otherwise the da-ecard $(n_1+n_2+2\sum\limits_{i=2}^t m_i,K)$ uniquely determines $G.$
\\
\textit{Case~$1.$}~$k_j>3,~n_1=1$ and $n_2\leq4.$\\
\indent Then $G=B(n_1,n_2,P_2,P_{k_j}).$
Consider a leaf da-ecard $(n_2+\sum\limits_{i=2}^t m_i,L)$ and a hub da-ecard $(n_1+n_2+2\sum\limits_{i=2}^t m_i,K).$
The leaf da-ecard forces $G$ to have a cycle of length $k_j.$
Hence in the hub da-ecard $(n_2+3,K),$ the newly added edge must be incident to the two bases and thus $dern(G)=2.$\\
\textit{Case~$2.$}~$k_j>3,~n_1=1$ and $n_2\geq5.$\\
\indent The da-ecard $(n_2+\sum\limits_{i=2}^t m_i,L)$ uniquely determines $G.$\\
\textit{Case~$3.$}~$k_j>3$ and $n_1\geq2.$\\
\indent Here the da-ecard {$(n_1+n_2+2\sum\limits_{i=2}^t m_i,K)$} will determine $G$ uniquely.\\
\textit{Case~$4.$}~$k_j=3.$\\
\indent Here the da-ecard $(n_1+n_2+2\sum\limits_{i=2}^t m_i,K)$ will do so.
\end{proof}
\begin{thm} \label{5}
For $k_i>2~(1\leq i\leq t),~dern(B(n_1,n_2,m_1P_{k_1},m_2P_{k_2},\dots,m_tP_{k_t}))=1$ or $2.$
\end{thm}
\begin{proof}
\indent In view of Theorems \ref{1}, \ref{3} and \ref{4'}, we assume that $n_1\neq n_2$ and at least two $m_i$'s are nonzero.
Let $m=m_1+m_2+\ldots+m_t.$\\
\textit{Case~1.}~`$n_1+m\geq6$' or `$n_1+m\leq5$ and $n_2-n_1\geq4$' or `$n_1+m=5$ and $n_2-n_1=1$'.\\
\indent The da-ecard $(n_2+m-1,L)$ uniquely determines $G.$\\
\textit{Case~$2.$}~`$n_1+m=3$ and either $n_2=k_2-2=2$ or $n_2=k_1+1=4$' or  `$n_1+m=4,~n_2-n_1<4,~(n_1,n_2)\neq (2,4)$ and $k_1=3$'.\\
\indent Here the da-ecard $(n_2+m,K)$ will do so.\\
\textit{Case~$3.$}~`$n_1+m=4,~(n_1,n_2)=(2,4)$ and $k_2=4$' or `$n_1+m=5,~n_2-n_1=2$ or $3$ and $k_1=3$'.\\
\indent The da-ecard $(n_1+m,K)$ will do the job here.\\
 For all the remaining cases, a leaf and a middle da-ecards determine $G$ uniquely~(as in Theorem~\ref{4'}).
\end{proof}

\begin{thm}
The degree associated edge reconstruction number of strong double brooms is 1 or 2.
\end{thm}
\begin{proof} Follows by Theorems \ref{4} and \ref{5}. 
\end{proof}
\section{Adern of Strong Double Brooms}
\indent\indent  For $2\leq i\leq\lfloor{\frac{k}{2}}\rfloor,$ we denote the da-ecard obtained from $B(n,n,mP_k),$  by deleting an edge on $mP_k$ whose ends are at distance respectively $i-1$ and $i$ from a hub vertex, by $M_i'.$\\
 \indent {\it In all graphs shown in this section, the dashed edge in extensions denotes the edge whose removal results in a de-ecard in common with $G.$}
\begin{thm}\label{9}
\begin{equation*}adern(B(n,n,2P_k))=
									\begin{cases}
									5& \textit{if $n=1$ and $k\geq4$}\\
									2& \textit{if `$n=2,3$ and $k=3$' or `$n\geq3$ and $k=4$'}\\
									1& \textit{if $n\geq4$ and $k=3$}\\
									3& \textit{otherwise}\\
									\end{cases}
									\end{equation*}
									
\end{thm}
\begin{proof}

Let $D$ denote any da-ecard of $G.$\\
\textit{Lower bound}:~For $n=1$ and $k\geq4,$ the graph $H_1$ (Figure~$2$), obtained from $K$ by adding an edge $e_1$ joining {the} 1-vertex at distance $k$ from the $3$-base and {the} $2$-vertex at distance $2k-3$ from that 1-vertex, shares two leaf da-ecards and two hub da-ecards with $G.$
Hence $adern(G)\geq5.$\\
\indent For $k=3,$ the graph obtained from $L$ by adding an edge joining a 1-vertex and a $2$-vertex (when $n=2$) or two $2$-vertices (when $n=3$), shares a leaf da-ecard with $G.$ For $n=1$ and $k=3,$ the graph, obtained from $L$ by adding an edge joining an isolated vertex and a $2$-neighbour of the $3$-base, has two leaf da-ecards in common with $G,$ which gives the desired lower bound.\\
\indent For $n=2$ and $k\geq4,$ consider the graph $H_2$ (Figure~$2$) obtained from $L$ by adding an edge $e_2$ joining the 1-neighbour of the $3$-base and a $2$-vertex at distance two from that $3$-base.
Clearly $H_2$ shares two leaf da-ecards with $G.$ 
Hence $adern(G)\geq2.$
\begin{center}
\scalebox{1} 
{
\begin{pspicture}(0,-1.80811)(13.56,1.80811)
\usefont{T1}{ptm}{m}{n}
\rput(6.427344,-1.5802975){Figure~$2.$~The extensions}
\psline[linewidth=0.04cm](1.8769397,0.8081097)(2.6969397,0.7881098)
\psline[linewidth=0.04cm](1.8769397,-0.37189028)(2.6969397,-0.37189028)
\psline[linewidth=0.04cm,linestyle=dotted,dotsep=0.16cm](2.7969396,0.7881098)(3.6169395,0.7681097)
\psline[linewidth=0.04cm,linestyle=dotted,dotsep=0.16cm](2.7969396,-0.3918903)(3.6169395,-0.3918903)
\psline[linewidth=0.04cm](3.6969397,-0.3918903)(4.5169396,-0.3918903)
\psline[linewidth=0.04cm](1.0011584,0.26810977)(1.7611583,0.7881098)
\psline[linewidth=0.04cm,linestyle=dashed,dash=0.17638889cm 0.10583334cm](1.0211583,0.20810975)(1.7611583,-0.37189025)
\psline[linewidth=0.04cm](4.6211586,-0.37189025)(5.381159,0.18810976)
\psline[linewidth=0.04cm](5.4611588,0.18810976)(6.321158,0.18810976)
\psline[linewidth=0.04cm,linestyle=dashed,dash=0.17638889cm 0.10583334cm](0.06115841,0.20810975)(0.90115845,0.20810975)
\psbezier[linewidth=0.04,linestyle=dashed,dash=0.16cm 0.16cm](3.649896,0.7696726)(4.609896,1.5896726)(6.049896,1.1096725)(6.409896,0.14967254)
\psline[linewidth=0.04cm,linestyle=dashed,dash=0.16cm 0.16cm](3.689896,0.7696726)(4.529896,0.7696726)
\usefont{T1}{ptm}{m}{n}
\rput(5.672708,1.3796724){$e_1$}
\usefont{T1}{ptm}{m}{n}
\rput(3.1927083,-1.0203274){$H_1$}
\psline[linewidth=0.04cm](8.95694,0.8281097)(9.77694,0.8081097)
\psline[linewidth=0.04cm](8.95694,-0.35189027)(9.77694,-0.35189027)
\psline[linewidth=0.04cm,linestyle=dotted,dotsep=0.16cm](9.87694,0.8081097)(10.69694,0.78810966)
\psline[linewidth=0.04cm,linestyle=dotted,dotsep=0.16cm](9.87694,-0.37189028)(10.69694,-0.37189028)
\psline[linewidth=0.04cm](10.77694,-0.37189028)(11.596939,-0.37189028)
\psline[linewidth=0.04cm](8.081158,0.28810978)(8.841158,0.8081097)
\psline[linewidth=0.04cm](8.101158,0.22810975)(8.841158,-0.35189024)
\psline[linewidth=0.04cm](11.701159,-0.35189024)(12.461159,0.20810974)
\psbezier[linewidth=0.04,linestyle=dashed,dash=0.16cm 0.16cm](10.8098955,0.86967254)(11.789896,1.6296724)(13.129895,1.4496726)(13.489896,0.48967254)
\psline[linewidth=0.04cm,linestyle=dashed,dash=0.16cm 0.16cm](10.769896,0.78967255)(11.609896,0.78967255)
\usefont{T1}{ptm}{m}{n}
\rput(12.752708,1.6196724){$e_2$}
\usefont{T1}{ptm}{m}{n}
\rput(10.272708,-1.0003276){$H_2$}
\psline[linewidth=0.04cm](7.169896,0.54967254)(8.009896,0.22967254)
\psline[linewidth=0.04cm](7.129896,-0.07032746)(8.009896,0.22967254)
\psline[linewidth=0.04cm](12.449896,0.20967256)(13.409896,0.5096725)
\psline[linewidth=0.04cm](11.609896,0.78967255)(12.429896,0.24967255)
\pscircle[linewidth=0.04,dimen=outer,fillstyle=solid](0.08,0.22343607){0.08}
\pscircle[linewidth=0.04,dimen=outer,fillstyle=solid](0.96,0.22343607){0.08}
\pscircle[linewidth=0.04,dimen=outer,fillstyle=solid](1.84,-0.3765639){0.08}
\pscircle[linewidth=0.04,dimen=outer,fillstyle=solid](2.78,-0.39656392){0.08}
\pscircle[linewidth=0.04,dimen=outer,fillstyle=solid](1.82,0.8034361){0.08}
\pscircle[linewidth=0.04,dimen=outer,fillstyle=solid](2.76,0.78343606){0.08}
\pscircle[linewidth=0.04,dimen=outer,fillstyle=solid](3.6,0.7634361){0.08}
\pscircle[linewidth=0.04,dimen=outer,fillstyle=solid](4.54,0.78343606){0.08}
\pscircle[linewidth=0.04,dimen=outer,fillstyle=solid](6.38,0.22343607){0.08}
\pscircle[linewidth=0.04,dimen=outer,fillstyle=solid](5.4,0.18343608){0.08}
\pscircle[linewidth=0.04,dimen=outer,fillstyle=solid](3.64,-0.39656392){0.08}
\pscircle[linewidth=0.04,dimen=outer,fillstyle=solid](4.58,-0.41656393){0.08}
\pscircle[linewidth=0.04,dimen=outer,fillstyle=solid](7.12,0.5834361){0.08}
\pscircle[linewidth=0.04,dimen=outer,fillstyle=solid](7.12,-0.076563925){0.08}
\pscircle[linewidth=0.04,dimen=outer,fillstyle=solid](8.06,0.24343608){0.08}
\pscircle[linewidth=0.04,dimen=outer,fillstyle=solid](8.9,0.8234361){0.08}
\pscircle[linewidth=0.04,dimen=outer,fillstyle=solid](8.9,-0.35656393){0.08}
\pscircle[linewidth=0.04,dimen=outer,fillstyle=solid](9.84,0.8034361){0.08}
\pscircle[linewidth=0.04,dimen=outer,fillstyle=solid](9.84,-0.35656393){0.08}
\pscircle[linewidth=0.04,dimen=outer,fillstyle=solid](10.74,0.8234361){0.08}
\pscircle[linewidth=0.04,dimen=outer,fillstyle=solid](10.7,-0.3765639){0.08}
\pscircle[linewidth=0.04,dimen=outer,fillstyle=solid](11.64,0.78343606){0.08}
\pscircle[linewidth=0.04,dimen=outer,fillstyle=solid](11.66,-0.3765639){0.08}
\pscircle[linewidth=0.04,dimen=outer,fillstyle=solid](12.48,0.22343607){0.08}
\pscircle[linewidth=0.04,dimen=outer,fillstyle=solid](13.48,0.5634361){0.08}
\pscircle[linewidth=0.04,dimen=outer,fillstyle=solid](13.46,-0.11656392){0.08}
\end{pspicture} 
}
\end{center}
\textit{Upper bound}:~We proceed by thirteen cases and prove that the collection of da-ecards considered under each case determines $G$ uniquely.\\We first give a table of outcome proof for the sake of readability.
\begin{center}
\scalebox{0.75}
{
\begin{tabular}{|c|c|c|c|}\hline
$k$&$n$&Cases&$adern(G)$\\\hline
&$1$&$3$
&3\\\cline{2-4} 
$3$&
$2,3$&
$2,~3$&$2$\\\cline{2-4}
&$\geq4$&$1,~2$&1\\\hline
&$1$&$4,~5.1,~6,~7,~10,~12$&$5$
\\\cline{2-4}
$4$&$2$&$4,~5.1,~9,~10,~13$&$3$
\\\cline{2-4}
&$\geq3$&
$1,~4,~5.1,~8,~10,~13$&$2$\\\hline
&1& $4,~6,~7,~12$&$5$\\\cline{2-4}
$5$&$2$&$4,~5.2.1,~9,~13$&$3$\\\cline{2-4}
 &$\geq3$&
$1,~4,~5.2.1,~8,~13$&$3$\\\hline
&1& $4,~5.2.2,~6,~7,~11,~12$&$5$\\\cline{2-4}
$\geq6$&$2$&$4,~5.2,~9,~11,~13$&$3$\\\cline{2-4}
 &$\geq3$&
$1,~4,~5.2,~8,~13$&$3$\\\hline
\end{tabular}
}
\end{center}

Let $S$ consist of the specified number of da-ecards from $G.$\\
\textit{Case~1.}~Any $L$ for $n\geq 4.$\\
\indent In $L,$ exactly two vertices have degree sum $n+1$ and so $G$ can be determined uniquely.
\vspace{0.3cm}\\
\textit{Case~$2.$}~Any $K$ for $n\geq2$ and $k=3.$\\
\indent In $K,$ the ends of the newly added edge must be a 1-vertex and an $(n+1)$-vertex.
Since all 1-neighbours of the $(n+2)$-base are similar and there is a unique $(n+1)$-vertex, all extensions are isomorphic and they are $G.$
\vspace{0.3cm}\\
\textit{Case~$3.$}~For $k=3,~adern(G)=2$ when $n=2,3$ and $adern(G)=3$ when $n=1.$ \\
\indent For $n=2,3,$ the da-ecard $K$ uniquely determines $G$ by Case~$2.$ Further, the da-ecard $L$ forces every extension other than $G$ to have exactly one da-ecard isomorphic to $L$ and no more da-ecards of $G.$ Hence $G$ can be uniquely determined by $L$ along with one more da-ecard and $adern(G)=2.$\\
\indent For $n=1,$ any $L$ and $K$ together determine $G$ as in Case~3 of Theorem~\ref{1}.
Now we shall show that any three isomorphic da-ecards together determine $G.$
In $K,$ the newly added edge must be incident to a 1-vertex and a $2$-vertex.
The extension (other than $G$), obtained by adding an edge joining a 1-neighbour of the unique $3$-base to a $2$-vertex, has exactly one hub da-ecard and the extension, obtained by joining the 1-neighbour of {the} $2$-base to a $2$-vertex, has exactly two da-ecards isomorphic to $K.$ 
Every extension non isomorphic to $G$ of $L$ has exactly two da-ecards isomorphic to $L.$\vspace{0.3cm}\\
\textit{Case~$4.$}~For $k\geq4,$ $L$ and $M$ (or $K$) when $n=2,~3,$ and $L$ and $M$ when $n=1.$\\ 
\indent For $n=2,3,$ the da-ecard $M,$ or $K$ forces $G$ to be connected and hence in $L,$ the only possibility to join $e$ is joining the unique isolated vertex and the unique $(n+1)$-vertex.\\
\indent For $n=1,$ the da-ecard $L$ forces $G$ to have a cycle of length $2(k-1)$ and hence in $M,$ the only possibility to join $e$ is joining the two 1-vertices at distance $2k-3.$\vspace{0.3cm}\\
\textit{Case~$5.$}~`$K$ and $M_i,~(2\leq i\leq \lfloor{\frac{k}{2}}\rfloor)$ when $k>4,~n\geq2$ or $k=4$' and `$M_i$ and $M_j,~(2\leq i<j\leq \lfloor{\frac{k}{2}}\rfloor)$ when $k>5$'.\\
\indent \textit{Case~$5.1.$}~$k=4.$\\
\indent \indent The da-ecard $K$ forces $G$ to have every base with at most n neighbours of degree 1 and hence \indent in $M_i,$ the new edge $e$ must be joined two 1-neighbours of the bases with $n+1$ neighbours of \indent degree 1.\\
\indent \textit{Case~$5.2.$}~$k>4.$\\
\indent \textit{Case~$5.2.1.$}~$K$ and $M_i$ when $n\geq2.$\\
\indent \indent The da-ecard $M_i$ forces $G$ to have two $(n+2)$-vertices and hence in $K,$ the new edge $e$ must be \indent joined the unique $(n+1)$-vertex and some 1-vertex.
The extension other than $G,$ obtained from $K$ \break \indent by joining $e$ to a 1-neighbour of the $(n+2)$-base, has exactly one da-ecard isomorphic to $K$ (since \break \indent the removal of any edge other than $e$ results in a disconnected da-ecard or a da-ecard having two \indent bases of degree at least 3 at distance $2$) and has no middle da-ecards (since  the removal of any \indent edge  results in a disconnected da-ecard or a da-ecard with a 1-vertex at distance $k-2$ from the \indent nearest  $(n+2)$-base).\\ 
\indent \textit{Case~$5.2.2.$}~$M_i$ and $M_j~(i\neq j).$\\
\indent \indent The da-ecards $M_i$ and $M_j$ have 1-vertices at distance $l_1>1$ and $l_2\geq1,$ respectively from \indent the nearest $(n+2)$-base such that $l_1>l_2$ (say).
Every extension other than $G,$ of $M_i$ has \indent no more middle da-ecards, since the removal of any $2$-edge results in a disconnected da-ecard or a \indent 1-vertex at distance $l_1$ or $k-l_1-2$ from the nearest $(n+2)$-vertex.\vspace{0.3cm}\\
\textit{Case~$6.$}~$\alpha L,~\beta K,~\alpha,\beta\geq1,~\alpha+\beta=4$ and $D$ when $n=1$ and $k\geq4.$\\ 
\indent The da-ecard $L$ forces $G$ to have a cycle of length $2(k-1)$ and hence the only extension of $K$ non isomorphic to $G$ is obtained by adding an edge joining {the} 1-vertex at distance $k$ from the {unique} $3$-base and to a $2$-vertex at distance $2k-3.$
This extension has exactly two da-ecards isomorphic to $L$ (obtained by removing each pendant edge) and  exactly two da-ecards isomorphic to $K$ (obtained by removing the $3$-edge incident to base) and has no more da-ecards of $G.$\vspace{0.3cm}\\
\textit{Case~$7.$}~$K,~M_i$ and $D$ when $n=1$ and $k>4.$\\
\indent The da-ecard $M_i$ has a 1-vertex at distance $l>1$ from the nearest $(n+m)$-base.
The extension other than $G$ obtained from $M_i$ by joining $e$ with a 1-vertex at distance $l$ or $k-l-2$ from the nearest $3$-base and a 1-neighbour of the other $3$-base, then the resulting graph has exactly one da-ecard $M_i$ (obtained by removing $e$) and exactly one da-ecard $K$ (obtained by removing the $3$-edge (non adjacent to $e$) lying on the cycle which is incident to a 3-vertex whose neighbours are $2$-vertices).
The above extensions have no more hub da-ecards (since the removal of any $3$-edge results in a disconnected da-ecard or a da-ecard with no $3$-base or a da-ecard having no 1-vertex at distance $k-2$ from the nearest $3$-base), no da-ecards isomorphic to $L$ (since the removal of any $2$-edge results in a da-ecard with no cycle of length $2(k-1)$) and no more middle da-ecards (since the removal of any $2$-edge results in a disconnected da-ecard or in a da-ecard with a 1-vertex at distance $l$ or $k-l-2~(>1)$ from the nearest $(n+2)$-vertex).
For all other extensions, the resulting graph has exactly one da-ecard isomorphic to $M_i$ and has no more da-ecards of $G.$\vspace{0.3cm}\\
\textit{Case~$8.$}~$2L$ determine $G$ when $k\geq4$ and $n=3.$\\
\indent Every extension (other than $G$) of $L$  has exactly one da-ecard isomorphic to $L$ (obtained by removing $e$) as the removal of any other $4$-edge results in a da-ecard with a cycle of length less than $2(k-1).$\vspace{0.3cm}\\
\textit{Case~$9.$}~$2L$ and $D$ when $k\geq4$ and $n=2.$\\
\indent Every extension (other than $G$) of $L$  obtained by joining $e$ with the 1-neighbour of the unique $3$-base and with a $2$-vertex at distance two from the $3$-base has exactly two da-ecards $L$ (obtained by removing the edges, lying on the cycle $C_4,$ incident to the $3$-neighbour of the base) as the removal of any other $3$-edge results in a da-ecard with a cycle $C_4$ or two bases at distance less than $k-1$ and no da-ecards $K$ and $M$ (since the removal of any other edge results in a da-ecard with an isolated vertex).
For any other extension, the resulting da-ecard has no leaf da-ecards (since the removal results in a da-ecard having a cycle of length less than $2(k-1)$) and has no da-ecards $K$ and $M_i$ (since the removal of any other edge results in a da-ecard with an isolated vertex).\vspace{0.3cm}\\
\textit{Case~$10.$}~$2M_i$ when $k=4.$\\
\indent The extension non isomorphic to $G$ of $M_i$ has exactly one  middle da-ecard (obtained by removing $e$), since the removal of any other $2$-edge results in a disconnected da-ecard.\vspace{0.3cm}\\
\textit{Case~$11.$}~$2M_i$ and $D$ when $k>5.$\\
\indent Let $M_i$ have a 1-vertex at distance $l>1$ from the nearest $(n+2)$-base.
Every extension of $M_i$ other than $G,$ obtained by   joining $e$ with a 1-vertex at distance $l$ or $k-l-2~(>1)$ from the nearest $(n+2)$-vertex and a 1-neighbour of the same $(n+2)$-base, has exactly two middle da-ecards and the remaining extensions other than $G$ has exactly one middle da-ecard as the removal of any other edge results in a da-ecard with no 1-vertex at distance $k-l-2~(>1)$ or $l$ from the nearest $(n+2)$-vertex or $(n+3)$-vertex or two 1-vertices at distance $k-l-2~(>1)$ or $l$ from the same $(n+2)$-vertex. None of the extensions of $M_i$ other than $G$ have a leaf da-ecard (since the removal results in a da-ecard having a cycle of length less than $2(k-1)$) and hub da-ecard (since the removal of any $(n+2)$-vertex results in a disconnected da-ecard or a da-ecard with a 1-vertex at distance $k-l-2~(>1)$ or $l$ from the nearest $(n+2)$-vertex or $(n+3)$-vertex.\vspace{0.3cm}\\
\textit{Case~$12.$}~$3K$ when $n=1$ and $k\geq4.$\\
\indent Every extension of $K$ other than $G$ must be obtained by joining $e$ with a 1-vertex and a $2$-vertex.
Clearly every extension has at most two hub da-ecards as the resulting extension has exactly two $3$-edges or the removal of any $3$-edge other than $e$ results in a disconnected da-ecard or a da-ecard having a cycle or a base with {two} neighbours of degree 1 or two $3$-bases or no $3$-base.\vspace{0.3cm}\\
\textit{Case~$13.$}~$2K$ when $n\neq1$ and $k\geq4.$\\
\indent Every extension of $K$ other than $G$ obtained by joining $e$ with two $2$-vertices (when $n=2$) or a 1-vertex and an $(n+1)$-vertex.
Clearly every extension has exactly one hub da-ecard, since the resulting extension has exactly one $(n+2)$-edge or the removal of any $(n+2)$-edge {other} than $e$ results in a disconnected da-ecard or a da-ecard having a base with $n-1$ neighbours of degree 1 (when $n>2$) or a `$3$-base or $4$-base' with one 1-neighbour (when $n=2$) or two $3$-bases (when $n=2$), which completes the proof of the theorem.
\end{proof}
\begin{thm}\label{10}
For $m\geq3,$
\begin{equation*}adern(B(n,n,mP_k))=
									\begin{cases}
									3& \textit{if `$k\geq5$' or `$m=3,~k=4$ and $n=1$'}\\
									2& \textit{if `$m=k=n+1=3$' or `$m=k+1=n+3=4$' or `$m=k=n+2=3$'}\\
									~& \textit{`$m\geq k=4$' or `$m=k-1=3$ and $n\geq2$'}\\
									1& \textit{otherwise}\\
									\end{cases}
									\end{equation*}						
\end{thm}
\begin{proof}
\textit{Lower bound:}~For $k\geq5,$ the graph $H_1$ (Figure~$3$), obtained from $M_i$ by joining a new edge $e_1$ between two 1-vertices at distance $3,$ shares two middle da-ecards with $G$ and hence $adern(G)\geq3.$\\
\indent For $m=3,~k=4$ and $n=1,$ the graph $H_2$ (Figure~$3$), obtained from $L$ by joining a new edge $e_2$ between the 1-vertex and a $2$-vertex at distance $3,$ shares two leaf da-ecards with $G.$
Hence $adern(G)\geq3.$\\
\indent For `$m\geq k=4$' or `$m=k-1=3$ and $n\geq2$', the graph, obtained from $K$ by joining a new edge between two 1-vertices at distance $2,$ has a hub da-ecard in common with $G$ and so $adern(G)\geq2.$
\indent The graph, obtained from $L$ by joining a new edge between two $2$-vertices for $m=k+1=n+3=4$ or $m=k=n+1=3$ and joining the 1-vertex and a $2$-vertex for $m=k=n+2=3,$ has a leaf da-ecard in common with $G,$ which gives the desired lower bound.
\begin{center}
\scalebox{1} 
{
\begin{pspicture}(0,-1.7765542)(10.86,1.7365543)
\usefont{T1}{ptm}{m}{n}
\rput(0.5785608,-0.40877196){$e_1$}
\psline[linewidth=0.04cm](1.8357483,0.9796654)(2.655748,0.9596654)
\psline[linewidth=0.04cm,linestyle=dashed,dash=0.17638889cm 0.10583334cm](1.8357483,0.27966544)(2.655748,0.27966544)
\psline[linewidth=0.04cm](2.7557485,0.9596654)(3.575748,0.93966544)
\psline[linewidth=0.04cm,linestyle=dotted,dotsep=0.10583334cm](3.8241858,0.21966553)(4.224186,0.23966552)
\psline[linewidth=0.04cm](1.0757483,0.9596654)(1.7157483,0.9796654)
\psline[linewidth=0.04cm](1.0757483,0.93966544)(1.7557483,0.25966543)
\psline[linewidth=0.04cm](4.415748,0.93966544)(5.0957484,0.9196654)
\psline[linewidth=0.04cm](4.435748,0.25966546)(5.0957484,0.87966543)
\psbezier[linewidth=0.04,linestyle=dashed,dash=0.17638889cm 0.10583334cm](0.068022065,0.5422236)(0.16418579,-0.38033447)(2.4696646,-0.46033445)(2.7041857,0.19772445)
\psline[linewidth=0.04cm](1.0041858,0.95966554)(1.7241858,1.6196655)
\psline[linewidth=0.04cm](1.8041859,1.6196655)(2.6641858,1.6196655)
\psline[linewidth=0.04cm](4.404186,1.6596655)(5.124186,0.9396655)
\psline[linewidth=0.04cm](5.144186,0.95966554)(6.0041857,1.2396655)
\psline[linewidth=0.04cm](5.144186,0.9396655)(6.0441856,0.5596655)
\psline[linewidth=0.04cm](0.10418578,1.2596655)(1.0241858,0.95966554)
\psline[linewidth=0.04cm](1.0241858,0.95966554)(0.10418578,0.5996655)
\psline[linewidth=0.04cm](2.6957486,1.6196654)(3.515748,1.5996654)
\psline[linewidth=0.04cm,linestyle=dotted,dotsep=0.10583334cm](3.8441858,0.9396655)(4.244186,0.95966554)
\psline[linewidth=0.04cm,linestyle=dotted,dotsep=0.10583334cm](3.8041859,1.5996655)(4.204186,1.6196655)
\usefont{T1}{ptm}{m}{n}
\rput(3.1642776,-0.9087416){$H_1$}
\usefont{T1}{ptm}{m}{n}
\rput(5.8974023,-1.5487417){Figure~$3.$~The extensions $H_1$ and $H_2$}
\usefont{T1}{ptm}{m}{n}
\rput(7.478561,-0.26877195){$e_2$}
\psline[linewidth=0.04cm](8.595748,0.9796654)(9.415748,0.9596654)
\psline[linewidth=0.04cm](7.8357487,0.9596654)(8.475749,0.9796654)
\psbezier[linewidth=0.04,linestyle=dashed,dash=0.17638889cm 0.10583334cm](6.864186,0.8596655)(6.9841857,-0.04033447)(9.229664,-0.36033446)(9.464186,0.29772443)
\psline[linewidth=0.04cm](7.764186,0.95966554)(8.484186,1.6196655)
\psline[linewidth=0.04cm](8.564186,1.6196655)(9.424186,1.6196655)
\usefont{T1}{ptm}{m}{n}
\rput(8.684278,-0.8887417){$H_2$}
\psline[linewidth=0.04cm](9.424186,1.6196655)(10.124186,1.0196655)
\psline[linewidth=0.04cm](9.424186,0.95966554)(10.124186,0.9796655)
\psline[linewidth=0.04cm](6.904186,0.95966554)(7.784186,0.95966554)
\psline[linewidth=0.04cm](7.804186,0.9196655)(8.504186,0.37966555)
\psline[linewidth=0.04cm,linestyle=dashed,dash=0.16cm 0.16cm](8.544186,0.35966554)(9.444186,0.35966554)
\psline[linewidth=0.04cm](9.544186,0.41966555)(10.104186,0.9796655)
\pscircle[linewidth=0.04,dimen=outer,fillstyle=solid](9.44,0.95655423){0.08}
\pscircle[linewidth=0.04,dimen=outer,fillstyle=solid](10.78,0.9765542){0.08}
\pscircle[linewidth=0.04,dimen=outer,fillstyle=solid](0.12,1.2565542){0.08}
\pscircle[linewidth=0.04,dimen=outer,fillstyle=solid](0.08,0.5565542){0.08}
\pscircle[linewidth=0.04,dimen=outer,fillstyle=solid](1.72,1.5965542){0.08}
\pscircle[linewidth=0.04,dimen=outer,fillstyle=solid](2.64,1.6165543){0.08}
\pscircle[linewidth=0.04,dimen=outer,fillstyle=solid](3.52,1.6165543){0.08}
\pscircle[linewidth=0.04,dimen=outer,fillstyle=solid](4.38,1.6565542){0.08}
\pscircle[linewidth=0.04,dimen=outer,fillstyle=solid](1.8,0.9765542){0.08}
\pscircle[linewidth=0.04,dimen=outer,fillstyle=solid](2.72,0.95655423){0.08}
\pscircle[linewidth=0.04,dimen=outer,fillstyle=solid](1.8,0.2765542){0.08}
\pscircle[linewidth=0.04,dimen=outer,fillstyle=solid](2.72,0.2365542){0.08}
\pscircle[linewidth=0.04,dimen=outer,fillstyle=solid](3.62,0.9365542){0.08}
\pscircle[linewidth=0.04,dimen=outer,fillstyle=solid](3.66,0.2365542){0.08}
\pscircle[linewidth=0.04,dimen=outer,fillstyle=solid](4.36,0.9365542){0.08}
\pscircle[linewidth=0.04,dimen=outer,fillstyle=solid](4.38,0.2365542){0.08}
\pscircle[linewidth=0.04,dimen=outer,fillstyle=solid](5.14,0.9365542){0.08}
\pscircle[linewidth=0.04,dimen=outer,fillstyle=solid](6.02,1.2565542){0.08}
\pscircle[linewidth=0.04,dimen=outer,fillstyle=solid](6.04,0.5565542){0.08}
\pscircle[linewidth=0.04,dimen=outer,fillstyle=solid](6.82,0.9365542){0.08}
\pscircle[linewidth=0.04,dimen=outer,fillstyle=solid](7.8,0.95655423){0.08}
\pscircle[linewidth=0.04,dimen=outer,fillstyle=solid](8.5,1.5965542){0.08}
\pscircle[linewidth=0.04,dimen=outer,fillstyle=solid](9.36,1.6165543){0.08}
\pscircle[linewidth=0.04,dimen=outer,fillstyle=solid](8.52,0.9765542){0.08}
\pscircle[linewidth=0.04,dimen=outer,fillstyle=solid](8.52,0.3565542){0.08}
\pscircle[linewidth=0.04,dimen=outer,fillstyle=solid](9.46,0.3365542){0.08}
\pscircle[linewidth=0.04,dimen=outer,fillstyle=solid](10.14,0.9765542){0.08}
\end{pspicture} 
}
\end{center}
\textit{Upper bound:}~We proceed by ten cases and prove that the collection of da-ecards considered under each case determines $G$ uniquely. We first give a table of outcome proof for the sake of readability.
\begin{center}
\scalebox{0.75}
{\begin{tabular}{|c|c|c|c|c|}\hline
$k$&$m$&$n$&Cases&$adern(G)$\\\hline
&&$1$& $3,8$&$2$\\\cline{3-5} 
&$3$&$2$&$2,~6$&$2$\\\cline{3-5}
&&$\geq3$&$1,~2$&$1$\\\cline{2-5}
\raisebox{1.5ex}{$3$}&&1&$2,~6$&2\\\cline{3-5}
&\raisebox{1.5ex}{$4$}&$\geq2$&$1,~2$&1\\\cline{2-5}
&$\geq5$&$\geq1$&$1,~2$&$1$\\\hline
&&1&$3,~4,~7,~8,~9$&$3$\\\cline{3-5}
&3&$2$&$3,~4,~6,~8,~9$&$2$\\\cline{3-5}
&&$\geq3$&$1,~4,~8,~9$&$2$\\\cline{2-5}
\raisebox{1.5ex}{$4$}&&1&$3,~4,~6~,8,~9$&$2$\\\cline{3-5}
&\raisebox{1.5ex}{$4$}&$\geq2$&$1,~4,~8,~9$&$2$\\\cline{2-5}
&$\geq5$&$\geq1$&$1,~4,~8,~9$&$2$\\\hline
&&1&$3,~4,~5~(if~k>5),~7,~8,~10$&$3$\\\cline{3-5}
&$3$&$2$&$3,~4,~5~(if~k>5),~6,~8,~10$&$3$\\\cline{3-5}
&&$\geq3$&$1,~4,~5~(if~k>5),~8,~10$&$3$\\\cline{2-5}
\raisebox{1.5ex}{$\geq5$}&&1&$3,~4,~5~(if~k>5),~6,~8,~10$&$3$\\\cline{3-5}
&\raisebox{1.5ex}{$4$}&$\geq2$&$1,~4,~5~(if~k>5),~8,~10$&$3$\\\cline{2-5}
&$\geq5$&$\geq1$&$1,~4,~5~(if~k>5),~8,~10$&$3$\\\hline
\end{tabular}
}
\end{center}
\textit{Case~1.}~Any $L$ for $n+m\geq6.$\\
\indent In $L,$ exactly two vertices have degree sum $n+m-1$ and so $G$ can be determined uniquely.\vspace{0.3cm}\\
\textit{Case~$2.$}~Any $K$ for $n+m\geq5$ and $k=3.$\\
\indent In $K,$ the neighbour of the newly added edge must be a 1-vertex and an $(n+m-1)$-vertex.
Since all 1-neighbours of the $(n+m)$-base are similar and there is a unique $(n+m-1)$-vertex, every extension obtained is unique and is isomorphic to $G.$\vspace{0.3cm}\\
\textit{Case~$3.$}~`$L$ and $K$ (or $M_i$) when $n+m<6$ and $k>3$' and `$L$ and $K$ when $n+m<5$ and $k=3$'.\\
\indent Proof is similar to Case~$3$ of Theorem~\ref{1}.\vspace{0.3cm}\\
\textit{Case~$4.$}~$M_i$ and $K$ for $k>3.$\\
\indent The da-ecard $M_i$ forces $G$ to have two $(n+m)$-vertices.
Hence in $K,~e$ must be joined to the $(n+m-1)$-vertex and some 1-vertex.
The only extension non isomorphic to $G,$ obtained by joining $e$ to a 1-neighbour of the $(n+m)$-base, has exactly one da-ecard isomorphic to $K$ (obtained by removing $e$), no more da-ecards isomorphic to $K$ (since the removal of any $(n+m)$-edge results in a da-ecard with cycle of length less than $2(k-1)$ or has at  most ${m-1\choose2}-1$ cycles of length $2(k-1)$), no middle da-ecards (since the removal of any $2$-edge results in a disconnected da-ecard or a da-ecard with 1-vertex at distance $k-2$ from the nearest $(n+m)$-base) and no leaf da-ecards as the removal of any edge results in a da-ecards with at most ${m-1\choose 2}$ cycles of length $2(k-1).$\vspace{0.3cm}\\
\textit{Case~$5.$}~$M_i$ and $M_j$ for $k>5.$\\
\indent Proof is similar to Case~$5.2.2$ of Theorem~\ref{9}.\vspace{0.3cm}\\
\textit{Case~$6.$}~Two isomorphic $L$'s when `$m=3$ and $n=2$' or `$m=4$ and $n=1$'.\\
\indent Proof is similar to Case~$8$ of Theorem~\ref{9}.\vspace{0.3cm}\\
\textit{Case~$7.$}~For $m=3,~n=1$ and $k>3,$ three isomorphic $L$'s.\\
\indent The extension {(other than $G$) of $L,$} obtained by adding an edge joining {the} 1-neighbour of the $(n+m)$-base and a $2$-vertex at distance $2$ from the $(n+m)$-base, has exactly two da-ecards isomorphic to $L$ (obtained by removing the edges incident to $2$-neighbour of the $(n+m)$-vertex).
The above extension and any other extensions (non isomorphic to $G$) have no more leaf da-ecards, since the removal of any $3$-edge results in a da-ecard with cycle of length less than $2(k-1).$\vspace{0.3cm}\\
\textit{Case~$8.$}~Two isomorphic $K$'s when $k>3$ or $n+m<5$ and $k=3.$\\
\indent Every extension other than $G$ of $K,$ has exactly one hub da-ecard, since the removal of any $(n+m)$-edge other than $e$ results in a da-ecard with cycle of length less than $2(k-1)$ or at most ${m-1\choose 2}$ cycles of length $2(k-1).$\vspace{0.3cm}\\
\textit{Case~$9.$}~Two isomorphic $M_i$'s for $k=4.$\\
\indent Every extension other than $G$ of $M_i,$ has exactly one middle da-ecard, since the removal of any $2$-edge other than $e$ results in a da-ecard containing a cycle of length 3.\vspace{0.3cm}\\
\textit{Case~$10.$}~For $k\geq5,$ three isomorphic $M_i$'s.\\
\indent Proof is similar to Case~$11$ of Theorem~\ref{9}. This case completes the proof of the theorem.
\end{proof}

\noindent {\bf Acknowledgement}.  This work is supported by the Research Project EMR/2016/000157 awarded by the Science and Engineering Research Board, Department of Science and Technology, Government of India, New Delhi.\\

\textbf{\Large {References}}
\begin{spacing}{1}
\begin{enumerate}[\hspace{0.01cm}\text{[}1\text{]}]

\bibitem{AM2} P. Anusha Devi and S. Monikandan, Degree associated reconstruction number of
graphs with regular pruned graph, \emph{Ars Combin.} 134 (2017) 29--41.
\bibitem{AM3} P. Anusha Devi and S. Monikandan, Degree associated reconstruction numbers of total graph, \emph{Contrib. Discrete Math.} 12 (2) (2017), 77--90.
\bibitem{AM4} P. Anusha Devi and S. Monikandan, Degree associated Rreconstruction number of connected digraphs with unique end vertex, \emph{Australas. J. Combin.} 66 (3) (2016), 365--377.
\bibitem{AM1} P. Anusha Devi and S. Monikandan, Degree associated edge reconstruction number of graphs with regular pruned graph, \emph{Electron. J. Graph Theory Appl.} (EJGTA) 3 (2) (2015), 146--161.
\bibitem{sur1}  K. J. Asciak, M.A. Francalanza, J. Lauri and W. Myrvold, A survey of some open questions in reconstruction numbers,
\emph{Ars Combin.} 97 (2010), 443--456.
\bibitem{west} M.D. Barrus and D.B. West, Degree-associated reconstruction number of graphs, \emph{Discrete Math.} 310 (2010), 2600--2612. 
\bibitem{bondy2}  J.A. Bondy, A graph reconstructor's manual, in \emph{Surveys in Combinatorics
(Proc. 13th British Combin. Conf.) London Math. Soc. Lecture Note} Ser. 166
(1991), 221--252.
 
\bibitem{Htext}{F. Harary, \emph{Graph Theory}, Addison Wesley, Mass. 1969.}  
\bibitem {Hrec} F. Harary, On the reconstruction of a graph from a collection of subgraphs,  \emph{Theory of graphs and its applications}, (M.~Fieldler ed.) Academic Press, New York (1964), 47--52.
\bibitem {HPrn} F. Harary and M. Plantholt, The graph reconstruction number, \emph{J. Graph Theory} 9 (1985), 451--454.
\bibitem{kelly1}  {P. J. Kelly, On isometric transformations, PhD Thesis, University of Wisconsin–
Madison, 1942.}
\bibitem{m1} M. Ma, H. Shi, H. Spinoza and D. B. West, Degree-associated reconstruction parameters of complete multipartite graphs and their complements, Taiwanese J. Math. {\it Taiwanese J. Math.} 19 (4) (2015), 1271--1284.
\bibitem{m2} M. Ma, H. Shi and D. B. West, The adversary degree associated reconstruction Number of double-brooms, {\it J. Discrete Algorithms} 33 (2015), 150--159.
\bibitem{manvel}  B. Manvel, Reconstruction of graphs - Progress and prospects, \emph{Congr. Numer.} 63 (1988), 177--187.
\bibitem{Molerndis} R. Molina, The edge reconstruction number of a disconnected graph, \emph{J. Graph Theory} 19 (3) (1995), 375--384.
\bibitem{MSr} S. Monikandan and S. Sundar Raj, Degree associated edge reconstruction number, in: \emph{Combinatorial Algorithms}, in: Lect. Notes Comput. Sci., vol. 7643, Springer-Verlag, Berlin (2012) 100--109.
\bibitem{Iwoca} S. Monikandan, P. Anusha Devi and S. Sundar Raj, Degree associated edge reconstruction number of graphs, {\it J. Discrete Algorithms} 23 (2013), 35--41.

\bibitem{Rdi1} S. Ramachandran, On a new digraph reconstruction conjecture, \emph{J. Combin.} Theory Ser. B 31 (1981), 143--149.
\bibitem{Rdrn1} S. Ramachandran, Degree associated reconstruction number of graphs and digraphs, \emph{Mano. Int. J. Math. Scis.} 1 (2000), 41--53. 
\bibitem{Rdrn2} S. Ramachandran, Reconstruction number for Ulam's conjecture, \emph{Ars Combin.} 78 (2006), 289--296.


\bibitem{Stoc} P. K. Stockmeyer, The falsity of the reconstruction conjecture for tournaments, \emph{J. Graph Theory} 1 (1977), 19--25.
\bibitem{ulam}  S. M. Ulam, A collection of mathematical problems, \emph{ Interscience Tracts in Pure and
Applied Mathematics} 8, (Interscience Publishers, 1960).

\end{enumerate}
\end{spacing}

\end{document}